\documentclass[reqno, a4paper, 10pt, oneside]{amsart}

\usepackage{amsmath, amssymb, amsthm, 
mathrsfs}
\usepackage{txfonts}
\usepackage{enumerate}
\usepackage{color}
\usepackage{bbm}

\usepackage{graphicx}
\usepackage{caption}
\usepackage[justification=centering]{subcaption}
\makeatletter
  \def\p@subfigure{\thefigure\,}
\makeatother

\usepackage{float} 

\usepackage{xcolor} 

\newtheorem{theorem}{Theorem}

\newtheorem{thm}{Theorem}
\newtheorem{cor}[theorem]{Corollary}
\newtheorem{dfn}[theorem]{Definition}
\newtheorem{lem}[theorem]{Lemma}
\newtheorem{prp}[theorem]{Proposition}

\theoremstyle{remark}

\numberwithin{equation}{section}
\numberwithin{theorem}{section}

\usepackage{scalerel}
\newcommand{\OB}{Ob{\l}{\'o}j} 
\newcommand{\F}{\mathcal{F}}

\newcommand{\E}{\mathbb{E}}
\newcommand{\N}{\mathbb{N}}

\newcommand{\R}{\mathbb{R}}

\newcommand{\G}{\mathcal G}

\newcommand{\seq}{\subseteq}

\renewcommand{\phi}{\varphi}

\renewcommand{\P}{P}

\newcommand{\sep}{\textrm{(SEP$_{\lambda, \mu}$)}}
\newcommand{\optsep}{\textrm{(OptSEP$_{\lambda, \mu}$)}}
\newcommand{\SG}{\textsf{SG}_{\gamma}}

\begin{document}
\title[Perkins Embedding for General Starting Laws]{Perkins Embedding for General Starting Laws}
\author{Annemarie Grass}\thanks{We acknowledge support by the Austrian Science Fund FWF through projects Y0782 and P35197. }
\thanks{Email: annemarie.grass@univie.ac.at}
\date{\today}
\begin{abstract}
The \emph{Skorokhod embedding problem} (SEP) is to represent a given probability measure as a Brownian motion $B$ at a particular stopping time. 
In recent years particular attention has gone to solutions which exhibit additional optimality properties due to applications to martingale inequalities and robust pricing in mathematical finance. 

Among these solutions, the \emph{Perkins embedding} sticks out through its distinct geometric properties. 
Moreover is the only optimal solution to the SEP which so far has been limited to the case of Brownian motion started in a dirac distribution. 

In this paper we provide for the first time an optimal solution to the Skorokhod embedding problem for the general SEP which leads to the Perkins solution when applied to  Brownian motion with start in a dirac. 
This solution to the SEP also suggests a new geometric interpretation of the Perkins solution which  better clarifies the relation to other optimal solutions of the SEP.
\end{abstract}
\maketitle
%
%
%
%
%
%
%
%
%
%
\section{Introduction}
Let $\mu$ be a probability measure on the real line with finite second moment 
and write $(B_t)_{t \geq 0}$ for a Brownian motion started according to a probability measure $\lambda$, i.e.\ $B_0 \sim \lambda$.  
The Skorokhod embedding problem (\cite{Sk61, Sk65}) is to find a stopping time $\tau$ such that the measure $\mu$ can be represented in the sense that
\begin{equation}
	B_{\tau} \sim \mu, \quad \text{ and } \E[\tau]<\infty.	\tag{SEP$_{\lambda,\mu}$}
\end{equation}
Skorokhod \cite{Sk61, Sk65} gave a solution in the case where Brownian motion is started in an atom.  A  necessary and sufficient condition for the SEP$_{\lambda, \mu}$ to admit a solution is for $\lambda$ to be prior to $\mu$ in convex order, i.e.\ 
$\int f(x) \lambda(dx) \leq \int f(x) \mu(dx)$ for any convex function $f: \R \rightarrow \R$. This follows by combining Skorokhod's results with Strassen's theorem   \cite{St65} on the existence of martingales with given marginals. 
We will tacitly assume this condition throughout the paper. The condition $\E[\tau]<\infty$ is imposed to exclude trivial solutions. 

Skorokhod's work initiated an active field of research. \OB's survey  \cite{Ob04} provides a comprehensive overview of the developments up to 2004 and describes more than 20 different solutions to the Skorokhod problem given by different authors.

During the 2000's a particular stimulous for the field has come from the connection with robust finance which was discovered in Hobson's seminal  paper \cite{Ho98a}, see also \cite{Ho11}. 
In view of these applications, it is of particular importance to construct stopping times which optimize certain  functionals subject to satisfying the embedding constraint. 
This question is known as the \emph{optimal Skorokhod embedding problem} and has been considered extensively, see \cite{ CoHo07, CoWa12, CoWa13, GaMiOb14, ObDoGa14, HeObSpTo16, ObSp16, BeCoHu14, BeCoHu16, GhKiPa19, CoObTo19, BaBe20, GhKiLi20a, BaZh21, GaObZo21, BeNuSt19} among others.
%
%
%
%
%
%
\subsection*{Perkins Embedding}
In this article we focus on a particular solution to \sep\, established by Perkins \cite{Pe86} for the \emph{deterministic start case} where $B_0 = 0$, i.e.\ $\lambda = \delta_0$.  The Perkins embedding solves an optimal Skorokhod embedding problem:
It has the characteristic optimality property of minimizing the law of the running maximum $\overline B_{\tau} := \max_{t \leq \tau}B_t$ of the underlying Brownian motion while simultaneously maximizing the law of the running minimum $\underline B_{\tau} := \min_{t \leq \tau}B_t$ among all solution to \sep. 
Here (and below) maximization / minimization of laws is understood with respect to first order stochastic dominance.  

The Perkins embedding is of importance for  robust pricing of barrier and lookback options, see \cite{BrHoRo01, Ho98a, Ho11}. 
While other solutions to the optimal Skorokhod embedding problem have  been  given in  the case of a general starting distribution $B_0\sim \lambda$ or admit direct  extensions to this  case, there is no known solution in the general starting case which extends the Perkins embedding.  
\begin{figure}
\centering
\begin{subfigure}{.45\linewidth}
\centering
\includegraphics[width = 1\linewidth]{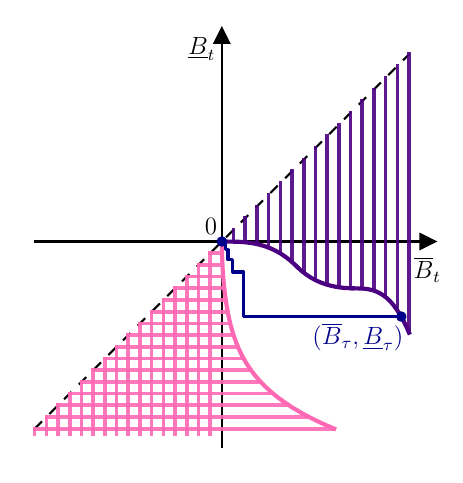}
\caption{Perkins solution in the $\lambda = \delta_0$ case.}
\label{fig:PerkinsTrivial}
\end{subfigure}
\quad
\begin{subfigure}{.45\linewidth}
\centering
\includegraphics[width = 1\linewidth]{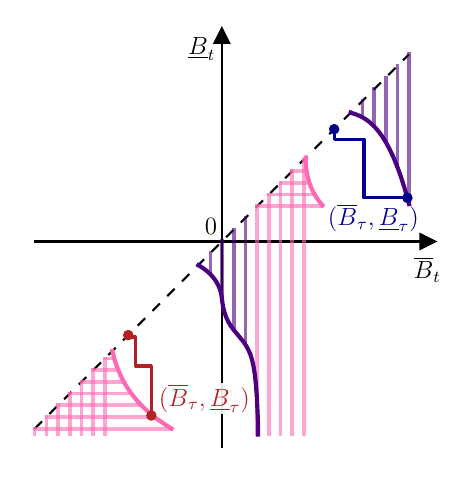}
\caption{Perkins solution with general starting.}
\label{fig:PerkinsRandom}
\end{subfigure}
\caption{Examples of Perkins solutions to \sep.}
\label{fig:Perkins}
\end{figure}

In \cite{HoPe02} Hobson and Pedersen propose a solution to \sep\, allowing for a general starting law which minimizes the law of the running maximum. Remarkably, the Hobson-Pederson embedding bears some resemblance to the Az\'ema-Yor \cite{AzYo79} embedding. As pointed out in \cite{HoPe02}, the Hobson-Pederson embedding does not maximize the running minimum in the $B_0=0$ case and specifically differs from the Perkins embedding also in this case. A further difference between the Hobson-Pederson embedding and the Perkins embedding, is that the latter does not require external randomization (unless $\mu(0)>0$, 
see Section \ref{s:Perkins-RandomStarting}).

Of specific interest is the geometric structure of Perkins solution. 
In the $\lambda = \delta_0$ case it can be identified as a hitting time of the process $(\overline B, \underline B)$ 
of a specifically structured target set $R \subseteq \R^2$, see Figure \ref{fig:Perkins} (A). 
The set $R$ is a union of `lines' of two types: 
\begin{itemize}
\item[(v)] \emph{vertical lines} which start below the diagonal and terminate in the diagonal. These lines will be depicted in `\emph{v}iolett'. 
\item[(h)] \emph{lines} which start on the right of the diagonal. They move horizontally until being reflected by the diagonal and continue to move vertically towards $-\infty$. These lines will be depicted in `\emph{h}ot pink'.
\end{itemize}
We will say that $R$ has \emph{vh-barrier structure}. 
Note that traditionally the Perkins picture is drawn without the horizontal lines being reflected downwards. 
These downward lines are irrelevant  in the deterministic starting case however become crucial when allowing for general starting as we will see below.

Our main contribution is to extend the Perkins solution to the case of general starting law. 
Specifically, setting 
\(
    (\lambda \wedge \mu) (A) := \inf_{\scaleto{B \seq A,\, B \text{ measurable}}{4pt}} \big (\lambda(B) + \mu(A \setminus B) \big) 
\) we have:
\begin{thm} 
\label{thm:maintheorem}
There exists a  vh-barrier  $R \seq \R^2$ such that the stopping time $\tau_{P}$ defined by 
\(P[\tau_{\text{P}} = 0, B_0 \in A]  = (\lambda \wedge \mu) (A)\) for Borel $A\subseteq \R$ and on $\{ \tau_{P} > 0 \}$ by 
\begin{equation}\label{eq:PerkinsTau}
    \tau_{P}= \inf \left\{t \geq 0 : \left(\overline B_{t}, \underline B_{t}\right) \in R \right\}
\end{equation}
is a solution to \sep. Moreover $\tau_{P}$  minimizes the law of the running maximum $\overline B$ and further maximizes  the law of the running minimum 
$\underline B$  among all  these solutions.

If $\lambda$ and $\mu$ are mutually singular measures we have $\tau_{P} > 0$ almost surely 
and $\tau_{P}$ is adapted to the filtration generated by the underlying Brownian motion.

The solution $\tau_{P}$ will be unique in the sense that if $\tilde \tau$ is another stopping time of the form \eqref{eq:PerkinsTau} then we have $\tau_{P} = \tilde \tau$ almost surely. 
\end{thm}
A possible realization of a general starting Perkins solution can be seen in Figure \ref{fig:Perkins} (B). 
We point out two important properties of this solution. 
Firstly  external randomization is only needed at time 0 and only if the two measure $\lambda$ and $\mu$ share mass. 
This answers a question raised in    \cite[Remark 2.3]{HoPe02} concerning the existence of adapted solutions in this setting. 
Secondly this solutions features the same representation as the hitting time of the process $(\overline B, \underline B)$ as the original Perkins solution and recovers it in the $\lambda = \delta_0$ case as illustrated in Figure \ref{fig:Perkins}.

We will focus on discussing the Perkins solution from a \emph{barrier type solution} viewpoint.
%
%
%
%
%
%
\subsection*{Barrier Type Solutions} 
\label{s:barriertypesolutions}
A \emph{barrier type solution} to \sep\, is traditionally of the form
\begin{equation}\label{eq:barriertypesolution}
    \tau_R = \inf \left\{t \geq 0 : (A_t, B_t) \in R \right\}
\end{equation}
for a sufficiently regular processes $A_t$ and a Borel set $R \seq \R^2$ featuring some additional \emph{barrier} structure.

Barrier type solutions not only come with natural geometric interpretations, 
we also see that all these solution are adapted to the filtration of  the underlying Brownian motion (apart from potential additional randomization at time $0$ if $\lambda$ and $\mu$ share mass).
Furthermore, barrier type solutions feature the following intrinsic uniqueness property highlighted by  Loynes \cite{Lo70}: 
if $S \seq \R^2$ is another barrier such that $\tau_S$ solves \sep\, then we have $\tau_R = \tau_S$ almost surely.
Prototypical barrier type solutions and also the solutions to originally coin the notion of a barrier resp.\  inverse barrier 
are the solutions by Root \cite{Ro69} and Rost \cite{Ro71,Ro76}.
A (Root) \emph{barrier} is a set $R \seq \R_+ \times \R$ such that $(s,x) \in R$ implies $(t,x) \in R$ for all $t > s$
 while for an \emph{inverse barrier} $S \seq \R_+ \times \R$ we require that $(s,x) \in S$ implies $(t,x) \in S$ for all $t < s$.
The Root resp.\  Rost solution is then given as the hitting time
\[
    \tau_{\text{Root}} = \inf \left\{t \geq 0 : (t, B_t) \in R \right\} 
    \quad \text{resp.} \quad 
    \tau_{\text{Rost}} = \inf \left\{t \geq 0 : (t, B_t) \in S \right\}
\]
and is known for \emph{minimizing} resp.\  \emph{maximizing} $\E[\tau^2]$ among all solutions to \sep\, (see \cite{Ki72} and \cite{Ro76}).
To be precise, in the Rost case,  if the initial and the terminal law share mass, this shared mass might need to be stopped immediately,
i.e.\ \(
    P\left[\tau_{\text{Rost}} = 0, B_0 \in A\right]  = (\lambda \wedge \mu) (A). 
\)

Another barrier type solution to \sep\, worth mentioning here is the Az\'ema-Yor solution \cite{AzYo79} which is  
known for maximizing the law of the running maximum of the underlying Brownian motion. 
This solution is also given as the hitting time of a barrier in the Root sense, however of the joint process $(\max_{s \leq t}B_s, B_t)$, thus
\[
	\tau_{AY} =  \inf \left\{t \geq 0 : \left(\max_{s \leq t}B_s, B_t\right) \in R \right\}.
\]
A general starting law version was found by Hobson in \cite{Ho98b} and - as in the Root case - 
the barrier type representation of the solution did not change. 
See Figure \ref{fig:BarrierTypeSolutions} for examples of Root, Rost and Az\'ema-Yor barrier type solution.
\begin{figure}
\centering
\begin{subfigure}{.3\linewidth}
\centering
\includegraphics[width = 1\linewidth]{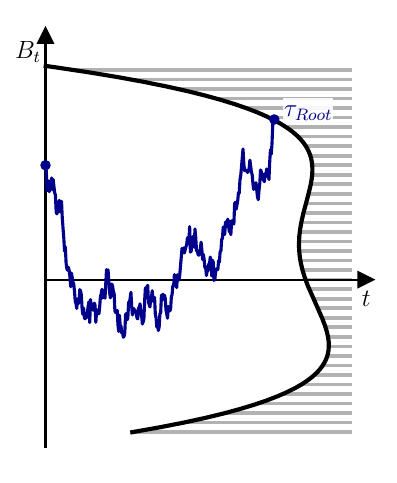}
\caption{Root solution.}
\label{fig:Root}
\end{subfigure}
\quad
\begin{subfigure}{.3\linewidth}
\centering
\includegraphics[width = 1\linewidth]{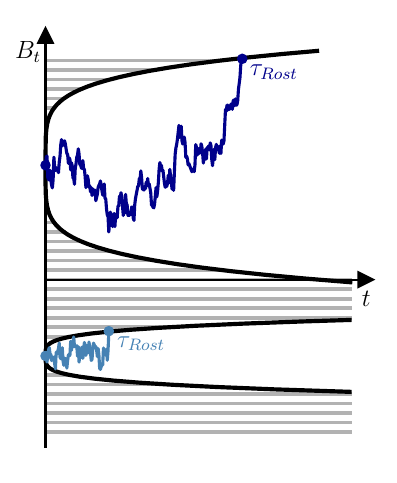}
\caption{Rost solution.}
\label{fig:Rost}
\end{subfigure}
\quad
\begin{subfigure}{.3\linewidth}
\centering
\includegraphics[width = 1\linewidth]{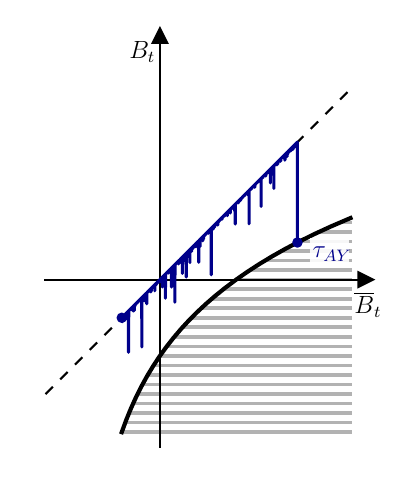}
\caption{Az\'ema-Yor solution.}
\label{fig:AzemaYor}
\end{subfigure}
\caption{Barrier type solutions to \sep.}
\label{fig:BarrierTypeSolutions}
\end{figure}
As both the Perkins solution and the Hobson-Pedersen solution \emph{minimize} the law of the running maximum of the underlying Brownian motion, 
they can be seen as a pendant to the Az\'ema-Yor embedding in a similar way as the Rost embedding can be seen as a pendant to the Root embedding.

Further barrier type solutions to \sep\, are e.g.\  the Vallois embedding \cite{Va83}, the Jacka embedding \cite{Ja88} or the cave embedding \cite{BeCoHu14}. We also refer to \cite{BeCoHu14} for a unifying framework for barrier type solutions. 

Notably it is not possible to find a process $(A_t)$ such that the Perkins solution can be represented as a barrier type solution in the classical sense \eqref{eq:barriertypesolution}, that is as a hitting time of the joint process $(A_t, B_t)$.
However, it is possible to identify it as a hitting time of the process $(\overline B, \underline B)$ of a vh-barrier as described above. 

In \cite{BeCoHu14} the authors use ideas and concepts of optimal mass transport to give existence results for solutions to \sep\, featuring additional extremal properties and moreover provide techniques on how connections between these additional properties of the solutions and their geometric characterizations as barrier type solutions can be made.
Minor variations of these techniques will allow us prove Theorem \ref{thm:maintheorem} and moreover to identify the right kind of phase space to interpret the Perkins embedding as the hitting time of a Rost inverse barrier, completing the Az\'ema-Yor$\leftrightarrow$Perkins -  Root$\leftrightarrow$Rost analogy described earlier. 
Furthermore this representation will enable us to prove a Loynes-type uniqueness result for  this solution.
%
%
%
%
%
%
\subsection*{Outline of the Article}
In Section \ref{s:MultifoldOptimalSEP} we recall  notation and results of \cite{BeCoHu14} and extend them to a multifold optimization problem over solutions to \sep.
These results will be applied in Section \ref{s:Perkins} 
to obtain existence of a Perkins embedding with general starting law as a solution to \sep\, with the desired additional extremal properties.
In Section \ref{s:uniqueness} we discuss uniqueness of barrier type solutions and propose a setting in which these results can be applied to the Perkins embedding.
This will conclude the proof of Theorem \ref{thm:maintheorem}.
%
%
%
%
%
%
%
%
%
%
%
%
\section{Multifold Optimal Skorokhod Embedding} \label{s:MultifoldOptimalSEP}
%
For this article we will adopt the underlying assumptions of \cite{BeCoHu14}, that is, we will work on a stochastic basis $\Omega= (\Omega, \G, (\G_{t})_{t \geq 0}, P)$ which is rich enough to support a Brownian motion $B$ and a uniformly distributed $\G_{0}$-random variable independent of $B$.
%
%
%
%
%
%
\subsection{Optimal Skorokhod Embedding}
%
The theory of finding solutions to the Skorokhod embedding problem featuring some optimality properties is presented from a systematic viewpoint  in \cite{BeCoHu14}, see also \cite{GuTaTo15b}. 
Historically it was also common to find a new solution to \sep\ and to only  later establish its additional optimality property. 
In \cite{BeCoHu14} this was turned around by imposing an optimization problem to the Skorokhod embedding problem and then to derive structural properties of the embedding problem from a `first order condition'.

For this we consider the set of \emph{stopped paths}.
\begin{dfn}[stopped paths]
We define the set of \emph{stopped paths} 
\[ 
S:= \left\{(f,s): s \geq 0, f \in C([0,s],\R) \right\}. 
\]
\end{dfn}
The optimal Skorokhod embedding is formulated in the following way:
\begin{dfn}[optimal Skorokhod embedding problem]
For a function
\[
	\gamma : S \rightarrow \R
\]
the \emph{optimal Skorokhod embedding problem} is to find a stopping time $\tau$ on $\Omega$ solving the following optimization problem
\[
	\inf \left\{ \E\left[\gamma((B_{t})_{t \leq \tau}, \tau)\right]: \tau \text{ solves \sep} \right\}. 								\tag{\text{OptSEP$_{\lambda, \mu}$}}
\]
\end{dfn}
We will always assume (OptSEP$_{\lambda, \mu}$) to be  \emph{well posed} in the sense that for all stopping times solving \sep, we have that $\E\left[\gamma((B_{t})_{t \leq \tau}, \tau)\right]$ exists, that $\E\left[\gamma((B_{t})_{t \leq \tau}, \tau)\right]\in (-\infty, \infty]$ and that there is at least one stopping time $\tau$ such that $\E\left[\gamma((B_{t})_{t \leq \tau}, \tau)\right]< \infty$.

The solutions presented in the introduction can be recovered as solutions to the optimal Skorokhod embedding problem in the following way.
The Root solution can be obtained by choosing $\gamma((f,s)):= s^2$, 
the Rost solution by $\gamma((f,s)):= -s^2$ and the Az\'ema-Yor solution by $\gamma((f,s)) := - \overline f = \max_{t \leq s} f(t)$. 
Many more classic solutions to the Skorokhod embedding problem can be represented this way, see \cite{BeCoHu14}.

As pointed out in the introduction, the Perkins solution stands out due to its two fold optimality property. 
The optimal Skorokhod embedding problem as described above will not suffice to describe this special embedding.
Also, in the construction of optimal solutions it is often useful to impose additional auxiliary  optimality conditions in order to guarantee uniqueness of solutions.
Thus we consider the following multifold optimal Skorokhod embedding problem.
%
\begin{dfn}[multifold optimal Skorokhod embedding problem]
Let $n \in \N$ and consider a function 
\[
\gamma=(\gamma_{1}, \dots, \gamma_{n}): S \rightarrow \R^{n}.
\]
Let $\mathrm{Opt}^{(1)}_{\gamma}$ be the set of all $\G$-stopping times $\tau$ on $\Omega$ solving the following optimization problem
\[
    \inf \left\{ \E\left[ \gamma_{1}((B_{t})_{t \leq \tau}, \tau)\right]: \tau \text{ solves \sep} \right\}. 						\tag{\text{OptSEP$^{(1)}_{\lambda, \mu}$}}
\]
\ \\ For $j \in \{2, \dots, n\}$ define $\mathrm{Opt}^{(j)}_{\gamma}$ as the set of all stopping times $\tau \in \mathrm{Opt}^{(j-1)}_{\gamma}$ that solve the optimization problem
\[
    \inf \left\{ \E\left[\gamma_{j}((B_{t})_{t \leq \tau}, \tau)\right]: \tau \in \mathrm{Opt}^{(j-1)}_{\gamma} \right\}.
			 			\tag{\text{OptSEP$^{(j)}_{\lambda, \mu}$}}
\]
\ \\ We call $(\mathrm{OptSEP}_{\lambda, \mu}) := (\mathrm{OptSEP}^{(n)}_{\lambda, \mu})$ the \emph{multifold optimal Skorokhod embedding problem}.
\end{dfn}
The optimization problem $(\mathrm{OptSEP}^{(1)}_{\lambda, \mu})$ is \emph{well posed} if for all stopping times solving \sep, 
 we have that $\E\left[\gamma_{1}((B_{t})_{t \leq \tau}, \tau)\right]$ exists, that $\E\left[\gamma_{1}((B_{t})_{t \leq \tau}, \tau)\right]\in (-\infty, \infty]$ and that there is at least one stopping time $\tau$ so that $\E\left[\gamma_{1}((B_{t})_{t \leq \tau}, \tau)\right]< \infty$.
We call $(\mathrm{OptSEP}_{\lambda, \mu})$ \emph{well posed}, if $(\mathrm{OptSEP}^{(1)}_{\lambda, \mu})$ is well posed and for all $j \in \{2, \dots, n \}$ the problem $(\mathrm{OptSEP}^{(j)}_{\lambda, \mu})$ is well posed in the above sense (considering stopping times in $\mathrm{Opt}^{(j-1)}_{\gamma}$).
%
%
%
%
%
%
\subsection{Randomized Stopping Times}
%
The requirement of solutions to \sep\, to be stopping times with respect to the filtration generated by Brownian motion is often too restrictive (as seen for the Hobson-Pedersen solution in the next section).

The key idea of linking the Skorokhod embedding problem to optimal transport is to think of a stopping time $\tau$ as a transport plan, mapping the mass of a trajectory $(B_t(\tilde \omega))_{t \geq 0}$ in the space  $(C(\R_+), \mathbb W_{\lambda})$ (where $\mathbb W_{\lambda}$ denotes the law of a Brownian motion started according to the distribution $\lambda$) to the endpoint $B_{\tau(\tilde \omega)}(\tilde \omega)$ in $\R$.
Even tough optimal transport theory cannot be applied directly, the appropriate analogues for this setup are developed in \cite{BeCoHu14}. 

The necessary relaxation of this problem is to  consider so called \emph{randomized stopping times} which can be seen as stopping times in the usual sense but on a probability space that is possibly enlarged.

The idea of randomized stopping times is made precise by defining them as a subset of subprobability measures on $C(\R_+) \times \R_+$, 
for details we refer to  \cite[Chapter 3.2]{BeCoHu14}.
\begin{dfn}[randomized stopping times]
A subprobability measure $\xi$ on $C(\R_+) \times \R_+$ is called a \emph{randomized stopping time} (of Brownian motion with initial distribution $\lambda$) if
	\begin{enumerate}[(i)]
		\item $\mathrm{proj}_{C(\R_+)} \leq \mathbb W_{\lambda}$.
		\item Given the disintegration $(\xi_{\omega})_{\omega \in C(\R_+)}$ of $\xi$ w.r.t.\ the first coordinate $\omega \in C(\R_+)$, $\tilde A_t(\omega) = \xi_{\omega}([0,t])$ is $\F_t ^a$-measurable for all $t \in \R_+$ (Here $(\F^a)$ denotes the natural augmented filtration on $C(\R_+)$).
	\end{enumerate}
For the probability measure $\mu$ on $\R$ we define the subset $\mathrm{RST}_{\lambda}(\mu)$ which consists of those $\xi \in \mathrm{RST}_{\lambda}$ such that
	\begin{enumerate}[(i)]
		\item $\mathrm{proj}_{C(\R_+)} = \mathbb W_{\lambda}$,
		\item For all $A \in \G$ we have $\xi(\{(\omega,t) \in C(\R_+) \times \R_+ : \omega(t) \in A \}) = \mu(A)$.
	\end{enumerate}
\end{dfn}
%
%
%
%
%
%
\subsection{Existence of an Optimizer}
%
One important result of \cite{BeCoHu14} is to establish  existence  of optimizers under fairly general conditions. 
Even though the main proofs were carried out for $n=1$ (Theorem 4.1) and generalizations are provided for $n=2$ (Theorem 6.1), we can generalize in precisely the same way to arbitrary $n \in \N$.
Let us formulate the optimal Skorokhod embedding problem for randomized stopping times and then give the existence results for this generalized problem.
\begin{dfn}[(OptSEP$\star_{\lambda, \gamma}$)]
Let $n \in \N$ and consider a function 
\[
\gamma=(\gamma_{1}, \dots, \gamma_{n}): S \rightarrow \R^{n}.
\]
Let $\mathrm{Opt}\star^{(1)}_{\gamma}$ be the set of all randomized stopping times solving the following optimization problem
\[
\inf \left\{  \int_{C(\R_+) \times  \R_+}   \gamma_1 ((\omega|_{[0,t]}, t)) \mathrm d \xi ((\omega, t)) : \xi \in \mathrm{RST}_{\lambda}(\mu) \right \}. 						\tag{\text{OptSEP$\star^{(1)}_{\lambda, \mu}$}}
\]
\ \\ For $j \in \{2, \dots, n\}$ define $\mathrm{Opt} \star ^{(j)}_{\gamma}$ as the set of all stopping times $\tau \in \mathrm{Opt} \star ^{(j-1)}_{\gamma}$ wich solve the optimization problem
\[
\inf \left \{  \int_{C(\R_+) \times  \R_+}   \gamma_j ((\omega|_{[0,t]}, t)) \mathrm d \xi ((\omega, t)): \xi \in \mathrm{Opt} \star^{(j-1)}_{\gamma} \right \}.
			 			\tag{\text{OptSEP$\star^{(j)}_{\lambda, \mu}$}}
\]
\ \\ We define $(\mathrm{OptSEP}\star _{\lambda, \mu}$):= $(\mathrm{OptSEP}\star ^{(n)}_{\lambda, \mu})$.
\end{dfn}
\begin{thm}[Existence of a minimizer, cf. \cite{BeCoHu14}, Theorem 4.1] 
\label{existenceTheorem}
Let $\gamma:S \rightarrow \R^{n}$ be lsc and bounded from below in the following sense:
\\ \indent For all $j \in \{1, \dots, n\}$ there exist  constants $a_{j},b_{j},c_{j} \in \R_{+}$ such that
	\begin{equation} \label{boundedness}
		-\left( a_{j}+b_{j}t+c_{j} \max_{s \leq t}B^{2}_{s}\right) \leq \gamma_{j}((B_{s})_{s \leq t},t)).
	\end{equation}
holds on $C(\R_+) \times \R_+$.
\ \\
 Then $\mathrm{(OptSEP}\star _{\lambda, \mu})$ admits a minmizer $\xi \in \mathrm{RST}_{\lambda}(\mu)$.
\end{thm}
The  following lemma provides equivalence between optimization over stopping times on our enlarged probability space and optimization over randomized stopping times.
\begin{lem}[cf. \cite{BeCoHu14}, Lemma 3.11]
Let $\tau$ be a $(G_t)_{t \geq 0}$-stopping time and consider the map
	\[
		\Phi: \Omega \rightarrow C(\R_+) \times \R_+, \quad	\overline \omega \mapsto ((B_t(\overline \omega))_{t \geq 0}, \tau(\overline \omega)).
	\]
Then $\xi := \Phi_{\#} \P \in \mathrm{RST}_{\lambda}$  and for any measurable map $\gamma: S \rightarrow \R$ we have
	\begin{equation}
		\int \gamma((f,s))d r_{\#} \xi ((f,s)) = \E_{\P} \left[\gamma((B_t)_{t \leq \tau}, \tau)\right]  \tag{$\star$}
	\end{equation}
for the function
	\[
		r: C(\R_+) \times \R_+ \rightarrow S, \quad (\omega, t) \mapsto (\omega|_{[0,t]}, t).
	\]
For any randomized stopping time $\xi \in \mathrm{RST}_{\lambda}$ we can find a $(G_t)_{t \geq 0}$-stopping time $\tau$ such that $\xi := \Phi_{\#} P$ and $(\star)$ holds.
\end{lem}
%
%
%
%
%
%
\subsection{Monotonicity Principle}
In optimal transport theory it was an important development to be able to identify optimal transport plans by the geometry of its support - see \emph{c-cyclical monotonicity}.
A monotonicity principle (`first order condition') similar to $c$-cyclical monotonicity for optimal Skorokhod embeddings is given in \cite{BeCoHu14}. 
We define the \emph{support} of a stopping time $\tau$ as a subset $\Gamma \seq S$ such that
	\[
		\P\left[((B_t)_{t \leq \tau}, \tau) \in \Gamma\right] =1.
	\]
Specific properties of this set $\Gamma$ will be crucial to our geometric approach to the Skorokhod embedding problem.

We will consider possible continuations of a given path and therefore define an operation of concatenation.
\begin{dfn}[concatenation of paths]
For two paths $(f,s),(g,t) \in S$ we define an operation of \emph{concatenation} $\oplus$ by 
\[
    (f \oplus g) ( r ):=    \begin{cases} f( r ) & r \in [0,s]
							    \\ f(s)-g(0)+g(r-s) & r \in (s, s+t].
							\end{cases}
\]
\end{dfn}
\begin{dfn}[going paths]
For any set of stopped paths $\Gamma \seq S$ we define the set of initial segments of this paths as 
\[ 
\Gamma ^{<}:= \left\{(f,s) \in S: \exists (\tilde f,\tilde s)\in \Gamma \text{ such that } s < \tilde s \text{ and } f|_{[0,s]}= \tilde f |_{[0,s]}\right\}
\]
and call it the set of \emph{going paths}.
\end{dfn}
\indent In the multifold optimal Skorokhod embedding problem we consecutively optimize over functions of stopped paths. 
We will soon learn that it is interesting and useful to derive structural arguments about the sets of stopped paths satisfying these optimality conditions. 
In order to do this, we would like to have a strategy for dealing with different paths stopping at the same value. 
Among those paths we would like to identify those paths that should be stopped and those paths that should be allowed to continue - keeping in mind our optimization problem. 
This leads us to the notion of \emph{stop-go} pairs, which by considering possible continuations of the paths gives a rule on how to decide on which one to stop and which one to allow to continue.

\begin{dfn}[Stop-Go Pair]
A pair of paths $((f,s), (g,t)) \in S \times S$ is called a \emph{stop-go pair} with respect to $\gamma$ (short: \textsf{SG}-pair) if
\begin{enumerate}[(i)]
		\item $f(s)=g(t)$
		\item $\E\left[\gamma(f \oplus (B_{u})_{u \leq \sigma}, s+\sigma)\right] + \gamma(g,t) 
                > \gamma(f,s)+ \E\left[\gamma(g \oplus (B_{u})_{u \leq \sigma}, t+\sigma)\right]$
			\ \\  in the lexicographic ordering of $\R^{n}$ for every $(\F_{t}^{B})_{t \geq 0}$-stopping time $\sigma$ such that $\E[\sigma] \in (0, \infty)$, both sides are well defined and the left-hand side is finite in every component.
\end{enumerate}
For the set of all \textsf{SG}-pairs we will write
\[
\textsf{SG}_{\gamma}:= \left\{((f,s), (g,t)) \in S \times S :((f,s), (g,t)) \text{ is a stop-go pair with respect to } \gamma  \right\},
\]
and for $j \in \{1, \dots, n\}$ we will call
\[
\E\left[\gamma_{j}(f \oplus (B_{u})_{u \leq \sigma}, s+\sigma)\right]+ \gamma_{j}(g,t) \geq \gamma_{j}(f,s)+ \E\left[\gamma_{j}(g \oplus (B_{u})_{u \leq \sigma}, t+\sigma)\right]
\]
the \emph{$j$-th stop-go condition} (\textsf{SGC$_{j}$}).
\end{dfn}

We now want to identify sets of stopped paths, such that it is not advantageous to stop any of these paths earlier in comparison to the other paths in this set. 
In other words, the stopping rule cannot be improved within this set.

\begin{dfn}[$\gamma$-Monotonicity]
A set of stopped paths $\Gamma \seq S$ is called \emph{$\gamma$-monotone} if 
\[
    \textsf{SG}_{\gamma} \cap(\Gamma^{<} \times \Gamma) = \emptyset 
\]
\end{dfn}
\begin{thm}[Monotonicity Principle] \label{monotonicity}
Let $\gamma:S \rightarrow \R^{n}$ be Borel measurable and let $\tau$ be a minimizer of $(\mathrm{OptSEP}_{\lambda, \mu})$. 
Then there exists a $\gamma$-monotone Borel set $\Gamma \seq S$ such that 
\[ 
    \P\left[((B_{t})_{t \leq \tau}, \tau) \in \Gamma\right] = 1.
\]
\end{thm}

Similar to the existence result also the monotonicity principle can be formulated for randomized stopping times.
As these technicalities are not important for the rest of this paper we will refer the reader to \cite{BeCoHu14} for further details.

We will see that the monotonicity principle is the key to the geometric approach to optimal Skorokhod embedding problems.
%
%
%
%
%
%
%
%
%
%
%
%
\section{The Perkins Embedding} \label{s:Perkins}
%
\subsection{Perkins Embedding with Deterministic Starting}
We will first give a precise formulation of the original Perkins solution in a barrier type formulation.
Finding a geometric interpretation of this solution in the case of $\lambda = \delta_{0}$ is feasible with the methods established in \cite{BeCoHu14}, see Theorem 6.8 therein.
 \ \\ 
We will give the following slight reformulation of this theorem in order to stress our interest in the specific structure of the target set.  
In addition to the barriers defined in the introduction we will furthermore consider an \emph{upwards barrier}, 
that is a set $R \seq \R^2$ such that $(s,x) \in R$ implies $(s,y) \in R$ for all $y>x$.

\begin{thm}[The Perkins embedding, cf. \cite{Pe86}] 
\label{Perkins0}
Let $\lambda= \delta_{0}$ and assume $\mu(\{0\})=0$. Let $\varphi: \R^{2}_{+}\rightarrow \R$ be a bounded continuous function which is strictly increasing in both arguments. Then there exists a stopping time $\tau_{P_0}$ which minimizes
\[
    \E\left[\varphi\left(\overline B_{\tau}, -\underline B_{\tau}\right)\right]
\] 
over all solution to \sep \,and which is of the form 
\[
    \tau_{P_0}= \inf \left\{ t \geq 0 : (\overline B_{t}, \underline B_{t}) \in R \right\}.
\]
Here $R \seq \R \times \R_{-}$ will be a vh-barrier which can be represented as $R = R_{1} \cup R_{2}$ with $R_{1}$ being an upwards barrier induced by (v)-lines and $R_{2}$ being an inverse barrier induced by (h)-lines. 
Moreover the boundaries of $R_{1}$ and $R_{2}$ are both given by decreasing functions $\R_{-} \rightarrow \R$ (see Figure \ref{fig:PerkinsTrivial}).
\end{thm}
The solution $\tau_{P_0}$ will in fact coincide almost surely for each choice of the auxiliary function $\varphi$ thus giving first order stochastic dominance. 

Unfortunately, some of the arguments in the proof of this theorem do not extend to the random starting case. 
Foremost it is no longer possible to optimize over the running minimum and the running maximum \emph{simultaneously}. 
The stopping rule of Perkins' problem will only stop paths when they reach a new running extremum. 
This justifies the representation of $\tau_{P_0}$ as a hitting time of the process $(\overline B, \underline B)$. 
Note that since we do not allow $\mu$ to hold any mass in $0$, we have $\tau_{P_0}>0$ and by properties of Brownian motion then 
$\overline B_{\tau_{P_0}}> \underline B_{\tau_{P_0}}$ a.s. 
These two facts imply that whenever we consider two paths stopped by this stopping rule at the same terminal value, either their running minima or their running maxima coincide. 
However, now it becomes very easy to decide on which one of those two paths to stop if we look at the other running extremum.

If we now allow for random starting we can also encounter (among other problems) the following situation: Two paths $(f,s), (g,t) \in S$ stop at the same value, that is $f(s)=g(t)$, however - lets say $f$ - does so by reaching a new running minimum and $g$ by reaching a new running maximum. 
Then $\overline f > \overline g$ and $\underline f > \underline g$. 
It is therefore no longer obvious, which of those two paths should be stopped and we can see that the solution can no longer be identified as the hitting time of a set serving both optimization problems simultaneously.
%
%
%
%
%
%
\newpage
\subsection{Perkins Embedding with Random Starting} \label{s:Perkins-RandomStarting}
%
%
\subsubsection*{The Hobson-Pedersen Solution} \label{s:HobsonPedersenSolution}
%
A first take on Perkins' embedding with random starting is due to Hobson and Pedersen in \cite{HoPe02}. 
We will give a brief sketch of their solution.

The authors define a function $g: \R \rightarrow \R$ and a random variable $G$ which is independent of the randomly started Brownian motion $(B_t)$. 
Both are explicitly determined by the measures $\lambda$ and $\mu$. 
They further define the two stopping times
\[
	\tau_G = \inf \left\{t > 0 : \overline B_t \geq G \right\} 
    \text{ and } 
    \tau_g = \inf \left\{t > 0 : B_t \leq g(\overline B_t) \right\},
\]
where one should note the resemblance to the Az\'ema-Yor solution of the second stopping time. 
The solution to the Perkins problem with random starting is then given by
	\[
		\tau_{HP} = \tau_G \wedge \tau_g,
	\]
see Figure \ref{fig:HobsonPedersen} for an illustration.
\begin{figure}
\centering
\includegraphics[width = 0.7\linewidth]{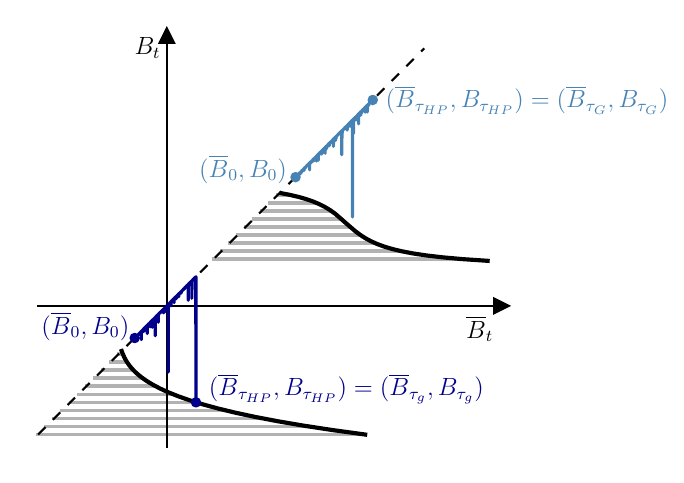}
\caption{The Hobson-Pedersen solution to \sep\, with non-deterministic starting law $\lambda$. 
We see the lower path being stopped by $\tau_g$ due to hitting the barrier and the upper path being stopped by $\tau_G$.}
\label{fig:HobsonPedersen}
\end{figure}
\begin{figure}
\centering
\begin{subfigure}{.45\linewidth}
\centering
\includegraphics[width = 0.8\linewidth]{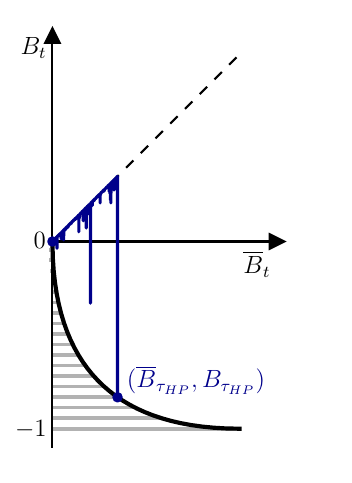}
\end{subfigure}
\quad
\begin{subfigure}{.45\linewidth}
\centering
\includegraphics[width = 0.8\linewidth]{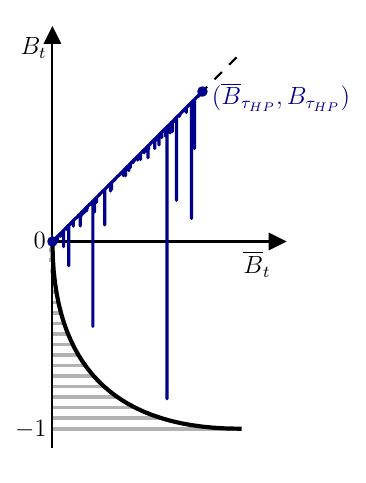}
\end{subfigure}
\caption{The Hobson-Pedersen solution to \sep\, for $\lambda = \delta_0 $ and $\mu = U[-1,1]$ (see Example A in \cite{HoPe02}). 
We see that external randomization is still needed (otherwise paths could only stop in $[-1,0]$) even though the measures $\lambda$ and $\mu$ are mutually singular.}
\label{fig:HP-Deterministic}
\end{figure}
Due to this external randomization given via the random variable $G$ this solution is not adapted to the filtration of the underlying Brownian motion. 
Moreover this external randomization generally remains to be needed in the deterministic case of $\lambda = \delta_0$ as seen in Figure \ref{fig:HP-Deterministic}.
In particular, the original solution by Perkins is not a special case of the Hobson-Pedersen solution.
The authors also discover that simultaneous optimization is no longer possible 
for general initial distributions and choose to prioritize minimizing the law of the running maximum.
%
%
%
\subsubsection*{Existence of a Perkins Embedding Allowing for a General Starting Law}
%
By interpreting Perkins' problem with general starting law as an \optsep \,for a well chosen $\gamma$ 
as introduced in Section \ref{s:MultifoldOptimalSEP} we can use the methods and results therein to prove existence of a solution 
that is given as the hitting time of the process $(\overline B, \underline B)$ of a specifically structured target set.

It was explained above why we can no longer optimize simultaneously over the running minimum and the running maximum. 
We will therefore decide - as in the Hobson-Pederson solution presented above - that from now on the running maximum is more important to us.
However, it should be obvious that all following results and calculations work analogously if our choice fell on the running minimum instead.

The following theorem provides a generalization of Perkins'  theorem in the slightly weaker formulation of optimization in expectation over an auxiliary function $\varphi$. 
While this is needed in order to apply the results of the previous section, the proof of Theorem \ref{thm:maintheorem} can be concluded in the next chapter by showing that all such solutions must coincide almost surely independently of the choice of $\varphi$.
Thus our solution will optimize over all such $\varphi$ which will conclude optimization in first order stochastic dominance. 

We no longer exclude the possibility of $\lambda$ and $\mu$ sharing mass. 
This leads to the situation of some randomization needed at time $0$.

\begin{thm}
\label{thm:Perkins}
Let $\varphi: \R \rightarrow \R$ be a continuous bounded strictly increasing function. 
Then there exists a stopping time $\tilde\tau$ which minimizes $\E\left[\varphi\left(\overline{B}_{\tau}\right)\right]$ over all solutions of $(\mathrm{SEP}_{\lambda, \mu})$ and maximizes $\E\left[\varphi\left(\underline{B}_{\tau}\right)\right]$ over all stopping times satisfying the former.
For $A \in \mathcal B (\R)$ we have 
	\[
		P\left[\tilde\tau = 0, B_0 \in A\right] = (\lambda \wedge \mu) (A)
	\]
and there exists a set $R \in \R^2$ such that on $\{ \tilde\tau > 0 \}$ we have
	\[
		\tilde\tau = \inf \left\{t \geq 0 : \left(\overline B_{t}, \underline B_{t}\right) \in R \right\}.
	\]
\end{thm}
\begin{proof}
We will consider the function $\gamma= (\gamma_{1}, \gamma_{2}, \gamma_{3}, \gamma_{4}): S \rightarrow \R^{4}$ given by
\begin{align*}
	\gamma_{1}((f,s))	&	:=  \varphi(\overline{f} ),				
\\	\gamma_{2}((f,s))	&	:= -\varphi(\underline{f}),				
\\  \gamma_{3}((f,s))	&	:= -\varphi( \overline f ) f(s)^{2},	
\\	\gamma_{4}((f,s))	&	:= -\varphi( -\underline f) f(s)^{2}	.
\end{align*}
Then all $\gamma_{j}$ are bounded from below in the sense of (\ref{boundedness}) (due to the boundedness of $\varphi$).
Thus Theorem \ref{existenceTheorem} guarantees the existence of a minimizer $\tilde\tau \in \mathrm{RST}_{\lambda}(\mu)$ which by Theorem \ref{monotonicity} is supported by a $\gamma$-monotone Borel set $\Gamma \seq S$.
\ \\
The stop-go conditions amount to the following:
\ \\ Let $((f,s), (g,t)) \in S \times S$ such that $f(s)=g(t)$, then
\begin{align*}
	\E\left[\varphi\left( \overline{f} \vee \left(f(s)+ \overline{B}_{\sigma}\right)\right)\right] + \varphi(\overline{g})
		&\geq \varphi( \overline{f}) + \E\left[\varphi\left( \overline{g} \vee \left(g(s) + \overline{B}_{\sigma}\right) \right)\right], 						\tag{SGC$_{1}$}
		\\ &
\\	\E\left[\varphi\left( \underline{f} \wedge \left(f(s)+ \underline{B}_{\sigma}\right) \right)\right] + \varphi( \underline{g})  
		&\leq \varphi( \underline{f} ) + \E\left[\varphi\left( \underline{g} \wedge \left(g(s) + \underline{B}_{\sigma}\right)\right)\right], 				\tag{SGC$_{2}$}
\end{align*}
\begin{align*}
		\E\left[\varphi\left(\overline{f} \vee \left(f(s) + \overline{B}_{\sigma}\right) \right)
         \left(f(s)+ B_{\sigma}\right)^{2}\right] &+  \varphi(\overline{g})g(t)^{2}  
		\leq
\\   \varphi(\overline{f})f(s)^{2} &+ \E \left[ \varphi\left( \overline{g} \vee \left(g(t) + \overline{B}_{\sigma}\right) \right) \left(g(t)+B_{\sigma}\right)^{2}\right], \tag{SGC$_{3}$}
\\ \E\left[ \varphi\left(-\underline{f} \wedge \left(f(s)+\underline{B}_{\sigma}\right) \right)
     \left(f(s)+ B_{\sigma}\right)^{2}\right] &+  \varphi(-\underline{g})g(t)^{2}  
		\leq
\\   \varphi(-\underline{f})f(s)^{2} &+ \E\left[ \varphi\left(-\underline{g} \wedge \left(g(t)
		+ \underline{B}_{\sigma}\right)\right)(g(t)+B_{\sigma})^{2}\right].  \tag{SGC$_{4}$}
\end{align*}

Let us first consider $\{ \tilde\tau > 0 \}$. 
By standard properties of Brownian motion we then may assume that for $(f,s) \in \Gamma$ we have $\overline f > \underline f$.

To legitimate that on $ \{ \tilde\tau > 0 \}$ the stopping time $\tilde\tau$ is given as the hitting time of the process $(\overline B, \underline B)$ we will start by showing that $\tilde \tau$ will only stop Brownian motion when it is reaching a new running minimum or running maximum.

Consider a path which stops somewhere between its current running extrema, that is a path $(f,s)\in S$ such that $\underline{f}< f(s)<\overline{f}$. Let $r$ be the time where $f$ hits its last new extremum. 
As $f$ does not reach a new extremum at time $s$ we can consider the initial segment of this path up to a time point $\tilde s < r$ such that $f(\tilde s)=f(s)$ and either $\underline f_{\tilde s}= \underline f$ or $\overline f_{\tilde s} = \overline f$, where
\[
		\overline{f}_{\tilde{s}}		 =\max_{u \leq \tilde{s}}f ( u )  \quad \text{ and} \quad		\underline{f}_{\tilde{s}}	 =\min_{u \leq \tilde{s}}f ( u ). 
\]
\ \\ Let $ \tilde f := f | _{[0, \tilde s]}$, then we claim that $((\tilde f, \tilde s),(f,s)) \in \SG $.
\ \\ \\ 1. Case: $\underline{f}_{\tilde s}=\underline{f}$ and $\overline{f}_{\tilde s}<\overline{f}$ as depicted in Figure \ref{fig:StopGo-max}, that is the last new extremum hit was a new maximum.
		\ \\ Assume $\overline{f}_{\tilde s} < f(\tilde s)+ \overline{B}_{\sigma}$, then $\overline{f}_{\tilde s} \vee \left( f(\tilde s)+ \overline{B}_{\sigma}\right)= f(\tilde s) + \overline B _{\sigma}$ and (SGC$_{1}$) reads
\begin{align*}
				\E\left[\varphi\left( f(\tilde s) + \overline{B}_{\sigma}\right)\right] + \varphi(\overline{f}) &\geq \varphi(\overline{f}_{\tilde s}) + \varphi( \overline{f} ) &\text{ if }\quad \overline{f} \geq f(s) + \overline{B}_{\sigma} 
		\\		\E\left[\varphi\left( f(\tilde s) + \overline{B}_{\sigma} \right)\right] + \varphi( \overline{f} ) &\geq \varphi( \overline{f}_{\tilde s} ) + \E\left[\varphi\left(f(s) + \overline{B}_{\sigma}\right)\right] & \text{ if }\quad \overline{f} < f(s) + \overline{B}_{\sigma}  
	\end{align*}
	and we see that in both cases a strict inequality holds due to our assumptions.  
		\ \\ On the other hand, if $\overline{f}_{\tilde s}\geq f(\tilde s) + \overline{B}_{\sigma}$, then $\overline{f}_{\tilde s} \vee \left(f(\tilde s)+ \overline{B}_{\sigma}\right) = \overline{f}_{\tilde s}$ as well as $\overline{f} \vee \left(f(s) + \overline{B}_{\sigma}\right)= \overline{f}$ and this leads to an equality in (SGC$_{1}$). 
Since $\underline{f}_{\tilde s}=\underline{f}$ we also have an equality in  (SGC$_{2}$) and therefore jump to our third condition (SGC$_{3}$):
\[
	\E\left[ \varphi(\overline{f}_{\tilde s})(f(\tilde s)+B_{\sigma})^{2}\right]+  \varphi(\overline{f})f(s)^{2} \leq  \varphi(\overline{f}_{\tilde s})f(\tilde s)^{2} +\E\left[ \varphi(\overline{f})(f(s)+B_{\sigma})^{2}\right]
\]
Again, due to our assumptions $f(\tilde s)=f(s)$, $\overline{f}_{\tilde s}<\overline{f}$ and as $\varphi$ is strictly increasing, a strict inequality holds and thus  $((\tilde f,\tilde s), (f,s)) \in \SG$.
\begin{figure}
\centering
\includegraphics[width = 0.8\linewidth]{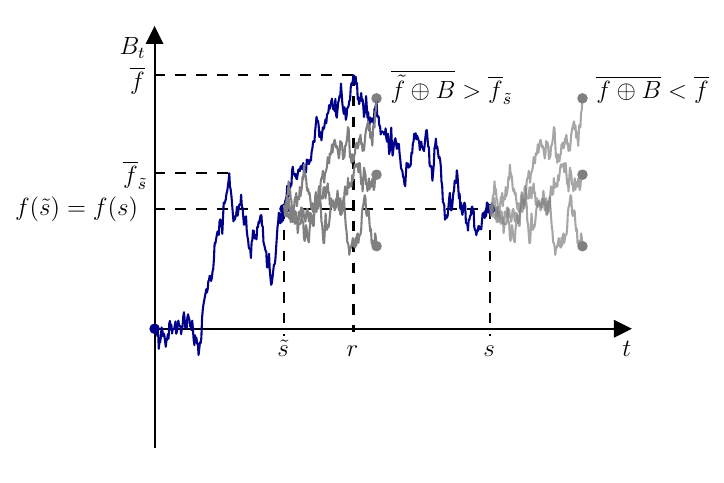}
\caption{Stop-Go pairs: We see possible continuations of our path $f$. 
If these continuations are attached at time $s$ the running maximum of $f$ remains unchanged. 
Attached at time $\tilde s$ however, the running maximum of $(\tilde f, \tilde s) = (f |_{[0, \tilde s]}, \tilde s)$ will be increased.}
\label{fig:StopGo-max}
\end{figure}

\ \\2. Case: $\overline{f}_{\tilde s}=\overline{f}$ and $\underline{f}_{\tilde s}>\underline{f}$, that is the last new extremum hit was a new minimum. Now (SGC$_{1}$) as well as (SGC$_{3}$) will always exhibit an equality and conditions (SGC$_{2}$) and (SGC$_{3}$) can be treated analogously to the previous case.
\ \\ \\
As now due to $\gamma$-monotonicity $((\tilde f, \tilde s), (f,s)) \not \in \Gamma^{<} \times \Gamma$ it follows that $\Gamma \cap  \left\{ (f,s) \in S : \underline{f}< f(s)<\overline{f} \right\} = \emptyset $, that is, when stopped we will almost surely have reached a new running minimum or a new running maximum.

Summing up, we now know that $\tilde\tau$ stops paths of Brownian motion only when they reach a new running minimum or a new running maximum.

We will denote the set of stopped paths satisfying this condition by 
\[
    \tilde S := \left\{ (f,s) \in S : f(s)= \underline f \text{ or } f(s)=\overline f \right\}.
\]
This justifies to consider the phase space $(\overline B, \underline B)$. 
All possible paths will lie below the diagonal, i.e.\ in the set $H_D:= \left\{(x,y) \in \mathbb{R}^2 : x \geq y \right\}$.

\ \\ We propose that there are two sets of points $r_{1}, r_{2} \seq H_D$, such that $R = R_{1} \cup R_{2}$, where
\[
R_{1}:= \big{\{} \{x \} \times [y,x] : (x,y) \in r_{1} \big{\}} 
\text{, and }
R_{2}:= \big{\{} \big{(}[y,x] \times \{y \} \big{)} \cup \big{(} \{y \} \times [y,\infty) \big{)}: (x,y) \in r_{2} \big{\}}.
\] 

Let us legitimate this target set structure:
\ \\ Assume we know that there is a path $(f,s) \in \Gamma $ such that $s>0$ and $f(s)= \overline f$, that is we stop at a new running maximum. We claim that it is impossible for a trajectory of the process $(\overline B, \underline B)$ to \emph{traverse}  the $\{ \overline f\} \times [\underline f, \overline f )$ line-segment and then be stopped as seen in Figure \ref{fig:StopGo} (A).

Consider a path $(g,t) \in \tilde S$ such that $g(0) \in (\underline f, \overline f]$, $\overline g \geq \overline f$ and $\underline g > \underline f$. Then there has to exist a timepoint  $\tilde t \leq t$ such that $g(\tilde t)= \overline g _{\tilde t} = \overline f=f(s)$ and note that still $\underline g _{\tilde t} > \underline f$ has to hold. However, this situation equals the 1.\ case of the above discussion, hence again $((g, \tilde t), (f,s)) \in \SG$. By $\gamma$-monotonicity it now follows that $(g, \tilde t) \not\in \Gamma^{<}$, therefore $(g,t) \not\in \Gamma$ and our claim is proven. 
\ \\ It becomes clear that by choosing $r_{1}:= \left\{(\overline f, \underline f) : (f,s) \in \Gamma \text{ and } f(s)= \overline f \right\} $ the definition of $R_{1}$ as above is reasonable.
\begin{figure}
\centering
\begin{subfigure}{.45\linewidth}
\centering
\includegraphics[width = 1\linewidth]{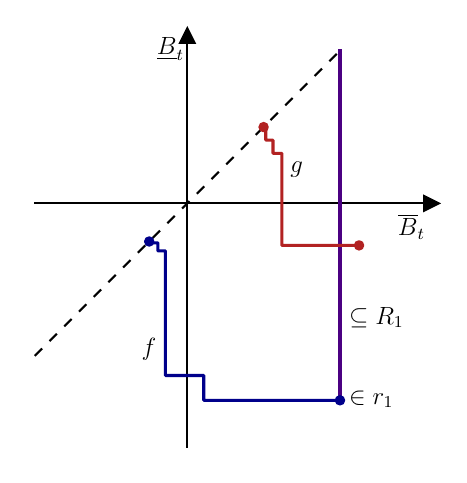}
\caption{The path $(f,s)$ is stopped at a new running maximum.}
\end{subfigure}
\quad
\begin{subfigure}{.45\linewidth}
\centering
\includegraphics[width = 1\linewidth]{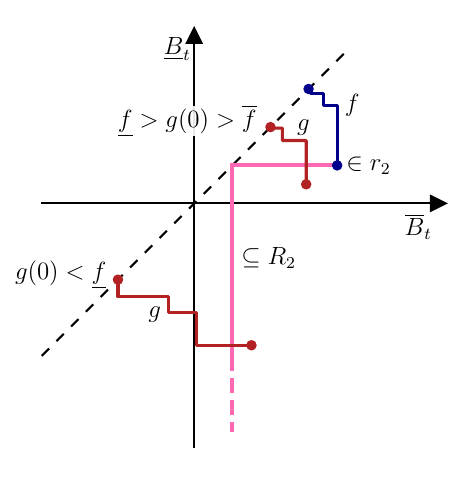}
\caption{The path $(f,s)$ is stopped at a new running minimum.}
\end{subfigure}
\caption{}
\label{fig:StopGo}
\end{figure}
\medskip
Now assume we know that there is a path $(f,s) \in \Gamma $ such that $s>0$ and $f(s)= \underline f$, that is we stop at a new running minimum. We claim that it is impossible for a trajectory of the process $(\overline B, \underline B)$ to \emph{traverse} either the $[\underline f, \overline f ) \times \{ \underline f\} $  or the $\{ \underline f\} \times [\underline f, \infty)$ line-segment and then be stopped as seen in Figure \ref{fig:StopGo} (B).

Consider a path $(g,t) \in \tilde S$ such that $ \overline g \in [\underline f, \overline f)$ and $\underline g \leq \underline f$. Let us first assume $g(0) \in [\underline f, \overline f)$.
Then again there has to exist a time point $\tilde t \leq t$ such that $g(\tilde t)=\underline g _{\tilde t} = \underline f =f(s) $ while $\overline g _{\tilde t} < \overline f$. Again $((g, \tilde t), (f,s)) \in \SG$ as in the 1.\ case above. 

Now assume $g(0) < \underline f$. It is then possible to find a timepoint $\tilde t < t$ such that $g(\tilde t) = \overline g _{\tilde t} = \underline f= f(s)$ and still $\overline g _{\tilde t}< \overline f$. 
We remember that (SGC$_{1}$) exhibits an equality if $\overline g _{\tilde t} \geq g(\tilde t) + \overline B _{\sigma}$, which in our setting is equivalent to $\overline B_{\sigma}=0$.
If on the other hand $\overline B_{\sigma}>0$, we will have a strict inequality.
However as trivial stopping times are excluded, the later will always happen with positive probability and thus taking the expectation (SGC$_{1}$) will always lead to a strict inequality, implying that $((g, \tilde t ), (f,s)) \in \SG$.

This concludes the proof of $(g,t) \not \in \Gamma$.
\ \\ Again, in analogy to the above let us take  $r_{2}:= \left\{(\overline f, \underline f) : (f,s) \in \Gamma \text{ and } f(s) = \underline f \right\}$ to justify the definition of $R_{2}$.

 \medskip
Now define the following two target sets
\begin{align*}
	R_{\textrm{CL}}	&:= R= R_{1} \cup R_{2},
\\	R_{\textrm{OP}}	&:= \big{\{} \{x \} \times (y,x] : (x,y) \in r_{1} \big{\}} 
 \cup \big{\{}\big{(}[y,x) \times \{y \} \big{)} \cup \big{(} \{y \} \times [y,\infty) \big{)} : (x,y) \in r_{2} \big{\}}.
\end{align*}
\noindent and consider
\[
    \tau_{\textrm{CL}}:= \inf \left\{t \geq 0 : (\overline B_{t}, \underline B_{t}) \in R_{\textrm{CL}} \right\} 
        \leq 
    \tau_{\textrm{OP}}:= \inf \left\{t \geq 0 : (\overline B_{t}, \underline B_{t}) \in R_{\textrm{OP}} \right\}.
\]
Note that since $\Gamma$ is Borel, the sets $r_1$ and $r_2$ are analytic sets since they are continuous images of Borel sets. 
This implies that $R_1$ and $R_2$ are analytic sets and we see that $\tau_{\textrm{CL}}$ and $\tau_{\textrm{OP}}$ are stopping times.

We would like to show that $\tau_{\textrm{CL}} = \tilde \tau =\tau_{\textrm{OP}}$ a.s.\ as our claim then follows.
\ \\
As $\Gamma \cap \left\{(f,s) \in S: \underline f < f(s) < \overline f \right\} = \emptyset $ it is obvious that $\tau_{\textrm{CL}}\leq \tilde \tau$ a.s.\  by definition of $\tau_{\textrm{CL}}$.

\medskip
To show that $\tilde \tau \leq \tau_{\textrm{OP}}$ a.s.\ let us assume that this is not the case. 
Choose 
$\omega \in \Omega$ such that $((B_{t}(\omega))_{t \leq \tilde \tau (\omega)}, \tilde \tau (\omega)) \in \Gamma$ 
and assume 
$\tau_{\textrm{OP}}(\omega) < \tilde \tau (\omega)$. 
Then there has to exist an $s \in \left[\tau_{\textrm{OP}}(\omega), \tilde \tau (\omega)\right)$ such that for $f=\left(B_{t}(\omega)\right)_{t\leq s}$ we have $(\overline f, \underline f) \in R_{\textrm{OP}}$. 
As $s \leq \tilde \tau (\omega)$ it follows, that $(f,s) \in \Gamma^{<}$, that is $(f,s)$ is a going path. 
By definition of $\tau_{\textrm{OP}}$ we can find find a point $(x,y) \in r_{1}$ such that $(\overline f, \underline f) \in \{ x\} \times (y, x]$ or a point $(x,y) \in r_{2}$ such that $(\overline f, \underline f) \in\big{(}[y,x) \times \{y \} \big{)} \cup \big{(} \{y \} \times [y,\infty) \big{)}$.
 However, we then find ourselves in the same situation of traversing line-segments as above. 
More precisely, by considering the path $(g,t) \in \Gamma$ corresponding to this $r_{1}$ respectively $r_{2}$ point we again find a $\textsf{SG}$-pair contradicting the $\gamma$-monotonicity of $\Gamma$. 
Hence $\tilde \tau \leq \tau_{OP}$ a.s.
\ \\ \indent
By standard properties of Brownian motion now $\tau_{\textrm{OP}}=\tau_{\textrm{CL}}$ a.s.\ which concludes the proof in the $\{ \tilde \tau > 0\}$ case.
\ \\
Let us now consider stopping in time $0$. 
Note that $P\left[\tilde \tau = 0, B_0 \in A\right] \leq (\lambda \wedge \mu) (A)$. 
We want to show that a strict inequality stands in conflict with the $\gamma$-monotonicity.
If $P\left[\tilde \tau = 0, B_0 \in A\right] < (\lambda \wedge \mu) (A)$ then there has to exists some $x \in A$ such that there are paths in $\Gamma$ starting in $x$ but not immediately stopping and also paths which stop in $x \in A$ at a strictly positive time. 
By above discussion we see that this would constitute $\textsf{SG}$-pairs and we can conclude 
	\[
		P\left[\tilde \tau = 0, B_0 \in A\right] = (\lambda \wedge \mu) (A). 
	\] \qedhere
\end{proof}
%
%
%
%
%
%
%
%
%
%
\section{Uniqueness} \label{s:uniqueness}
%
In the previous section we have seen that the Perkins solution with random starting is given as a hitting time of the process $(\overline B, \underline B)$ of a specifically structured target set. 
As mentioned in the introduction it was proved by Loynes \cite{Lo70} that Root's barrier type solution is essentially unique. 
In this chapter we will briefly discuss the extension of this argument to \emph{barrier type solutions} of the form
		\[
			\tau_R = \inf \left\{t \geq 0 : (A_t, B_t) \in R \right\}
		\]
for sufficiently regular processes $A_t$.

We will then propose a setting in which the Perkins solution with random starting 
can also be seen as a barrier type solutions and show how the Loynes argument extends to this setting. 
This will enable us to conclude the proof of Theorem \ref{thm:maintheorem}.

\subsection{The Loynes Uniqueness Result}
Root initially defined barriers as topologically \emph{closed} subsets of $\R$. 
Loynes uniqueness argument relies on this fact by using that we are actually inside the barrier when we stop.
A suitable generalization for our purposes would be to ask our process $A_t$ to be sufficiently regular such that $(A_t,B_t)$ is jointly Markov satisfying the Blumenthal-Getoor 0-1-law. 
Instead of topological closures we then consider \emph{fine closures} with respect to the process $(A_t,B_t)$.

The \emph{fine closure} of a set $R \seq \R^2$ with respect to a jointly Markov process $(A_t, B_t)$ will be denoted by 
$R^{*}$ and is defined as
\[ 
    R^{*}:=R \cup \left\{(t,x)\in \R^{2} : P\left[\tau_{R} =0\big{|} (A_{0},B_{0})=(t,x)\right]=1 \right\}. 
\]
By this definition follows $\tau_{R}=\tau_{R^{*}}$ a.s. and $(A_{\tau_{R^{*}}},B_{\tau_{R^{*}}}) \in R^{*}$ by Blumenthal-Getoor. 
Moreover we also have $A_{\tau_{R}} \sim A_{\tau_{R^{*}}}$ and $B_{\tau_{R}} \sim B_{\tau_{R^{*}}}$.
\ \\ For details on fine closures refer to \cite{ChWa05}, see also the arguments in \cite{BeCoHu14}.
\ \\ We see that taking fine closures does not alter the stopping properties and we will therefore assume without loss of generality that our barriers are always finely closed with respect to $(A_t,B_t)$.

\begin{prp} \label{prp}
Let $R$ and $S$ be two barriers such that  
	\[
		\tau_{R}:= \inf \left\{t \geq 0 : (A_{t},B_{t})\in R \right\}  \text{ and }  
        \tau_{S}:= \inf \left\{t \geq 0 : (A_{t},B_{t})\in S \right\}
	\]
 are stopping times both generating the same law $\mu$.
Then $R \cup S$ also generates $\mu$ and the corresponding stopping time is given by \[\tau_{R \cup S}= \tau_{R} \wedge \tau_{S}.\]
\end{prp}

\begin{proof}
\ Consider the set 
\[K:= \left\{z \in \R: \inf \left\{y \in \R : (y,z)\in R \right\} < \inf \left\{y \in \R : (y,z)\in S \right\}  \right\}, \]
and define 
\begin{align*}
	R_{K}:=\left\{(x,y) \in R : y \in K  \right\}, \quad R_{K^{c}}:=\left\{(x,y) \in R : y \in K^{c}  \right\},
\\ 	S_{K}:= \left\{(x,y) \in S : y \in K \right\}, \quad S_{K^{c}}:= \left\{(x,y) \in S : y \in K^{c} \right\}. 
\end{align*}

Note that $S_{K} \seq R_{K}$ as well as $R_{K^{c}} \seq S_{K^{c}}$.

Now assume that $B_{\tau_{S}}\in K$. 
By definition of $\tau_{S}$ and by the above discussion we have $(A_{\tau_{S}},B_{\tau_{S}})\in S_{K} \seq R_{K}$. 
This means, $\tau_{R} \leq \tau_{S}$ and  it is now impossible to have $(A_{\tau_{R}},B_{\tau_{R}})\in R_{K^c}$ (otherwise $B_{\tau_{S}}$ would have stopped in $K^{c}$). 
This implies that we cannot have $B_{\tau_{R}} \in K^{c}$ and therefore $P[B_{\tau_{S}} \in K, B_{\tau_{R}}\in K^{c}]=0$.
\ \\
As $B_{\tau_{R}} \sim B_{\tau_{S}}$ we have:
\begin{align*}
    P\left[B_{\tau_{R}}\in K \right]	
        &= P\left[B_{\tau_{R}}\in K, B_{\tau_{S}}\in K \right] + P\left[B_{\tau_{R}}\in K, B_{\tau_{S}}\in K^{c} \right]
\\ 		&= P\left[B_{\tau_{S}}\in K, B_{\tau_{R}}\in K \right] + P\left[B_{\tau_{S}}\in K, B_{\tau_{R}}\in K^{c} \right] 
         = P \left[B_{\tau_{S}}\in K \right],
\end{align*}
and altogether 
\[
    0 = P\left[B_{\tau_{R}}\in K, B_{\tau_{S}}\in K^{c}\right]
      = P\left[B_{\tau_{S}}\in K, B_{\tau_{R}}\in K^{c}\right].
\]
It is now clear that the sets
\[
    \Omega_{1} :=\left\{B_{\tau_{R}} \in K, B_{\tau_{S}} \in K \right\} \,\, \text{ where }\,\, \tau_{R}\leq \tau_{S} 
\]
and
\[
    \Omega_{2}:=\left\{B_{\tau_{R}} \in K^{c}, B_{\tau_{S}} \in K^{c} \right\} \,\, \text{ where }\,\, \tau_{S}\leq \tau_{R}
\]
are essentially disjoint and their union has full probability. 
Therefore $\tau_{R \cup S}= \tau_{R} \wedge \tau_{S}$ a.s.
\ \\
That $\tau_{R \cup S}$ generates the same law is now obvious (due to decomposition into $\Omega_{1}$ and $\Omega_{2}$).
	
\end{proof}

\begin{cor} \label{cor}
	If $\tau$ is a barrier type solution to the (SEP$_{\lambda, \mu}$), then $\tau$ is a.s. unique.
\end{cor}

\begin{proof}
For any solution $\tau$ to (SEP$_{\lambda, \mu}$) we have $\E[\tau] < \infty$ and thus $\E[B_{\tau}^2] = \E[\tau]$. 
Now if $\tilde \tau$ is another solution to (SEP$_{\lambda, \mu}$) and $\tilde \tau \leq \tau$, then 
$\E[\tilde \tau] = \E[B_{\tilde \tau}^2] = \E[B_{\tau}^2] = \E[\tau]$, hence $\tau = \tilde \tau$ a.s. and we call such a solution \emph{minimal}. 
Let $R$ be the barrier inducing the solution $\tau = \tau_R$ and let $\tau_S$ be a different solution induced by the barrier $S$. 
By Proposition \ref{prp} we have that $\tau_R \wedge \tau_S$ is also a solution to the (SEP$_{\lambda, \mu}$). 
Since trivially $\tau_R \wedge \tau_S \leq \tau_R$ we have that $\tau_R \wedge \tau_S = \tau_R$ due to the minimality observation before. 
Now analogously $\tau_R \wedge \tau_S = \tau_S$, hence $\tau_R=\tau_S$ a.s.\ concluding the proof of uniqueness.
\end{proof}

We see that this uniqueness result cannot immediately be applied to the solutions found in Theorem \ref{thm:Perkins}. 
However, an analogous uniqueness result holds in this setting. 
The target set constructed in the proof of Theorem \ref{thm:Perkins} can be seen as an inverse barrier in the sense of Rost and the Loynes type uniqueness result can be extended to this setting.
%
%
%
%
%
%
\subsection{Uniqueness of the Perkins Solution}
%
To give a Loynes type uniqueness result for the Perkins embedding we need to identify the right space and setting to be able to consider vh-barriers as barriers in the classic sense. 
We aim to represent vh-barriers as \emph{inverse} barriers in the Rost sense. 

Let $(\mathrm R', \preceq)$ be a disjoint copy of the real numbers where the order is flipped, 
i.e.\  for $x,y \in \mathrm R'$ we have $x \preceq y$ if and only if $x \geq y$ in the usual order of the real numbers. 
We keep $\mathrm R'$ distinguishable from $\mathbb{R}$. 
Consider $\mathrm D:= \mathrm R' \cup \mathbb{R}$ and define an order $\leq_{\mathrm D}$ on $\mathrm D$ in the following way.
\[
    x,y \in \mathrm D, \text{ then } x \leq_{\mathrm D} y :\Leftrightarrow 
        \begin{cases}
            x \in \mathrm R'          \text{ and } y \in \mathbb{R}
        \\ x,y \in \mathrm R'         \text{ and } x \preceq y
        \\ x,y \in \mathbb{R} \text{ and } x \leq y
        \end{cases}
\]
This order enables us to depict the set $\mathrm D$ as linearly ordered from left to right.  
Moreover we can consider $\mathrm R'$ as the \emph{left half} of $\mathrm D$ 
while considering $\mathbb{R}$ as the \emph{right half} of $\mathrm D$. 
To emphasize this we consider the map
\[
    \iota : \mathbb{R} \rightarrow \mathrm R', \quad \iota(x) = x
\]
which embeds $\mathbb{R}$ into $\mathrm R'$.  
Then 
\[
    \cdots \leq_{\mathrm D} \iota(1) 
           \leq_{\mathrm D} \cdots 
           \leq_{\mathrm D} \iota(0) 
           \leq_{\mathrm D} \cdots 
           \leq_{\mathrm D} \iota(-1) 
           \leq_{\mathrm D} \cdots  
           \leq_{\mathrm D} -1 
           \leq_{\mathrm D} \cdots  
           \leq_{\mathrm D} 0
           \leq_{\mathrm D} \cdots  
           \leq_{\mathrm D} 1 
           \leq_{\mathrm D} \cdots,  
\]
thus we may consider $\mathrm R'$ as $\mathbb{R}$ reflected in the origin. 

We will now construct \emph{barriers} in 
$\mathrm D \times \mathbb{R} = \left(\mathrm{R}' \times \mathbb{R} \right) \cup \left(\mathbb{R} \times \mathbb{R} \right)$. 
Note that in the proof of Theorem \ref{thm:Perkins} the vh-barrier $R$ was established as $R = R_1 \cup R_2$ where $R_1$ was induced by $r_1$, the endpoints of paths being stopped at a new running maximum while 
$R_2$ was induced by $r_2$, the endpoints of paths being stopped at a new running minimum.
While $R$ is a barrier in $\mathbb{R}^2$ we can also use $r_1$ and $r_2$ to define a corresponding barrier $R_{\mathrm D} \seq \mathrm D \times \mathbb{R}$ in the following way.
For every $(x,y) \in r_1$ we have $(x,y) \in \mathrm{R}' \times \mathbb{R} \seq \mathrm D \times \mathbb{R}$ and also 
$(\tilde x, y) \in R_{\mathrm D}$ for all $\tilde x \in \mathrm{D}$ such that $\tilde x \leq_{\mathrm{D}} x$.
Note that as $x \in \mathrm{R}'$ of course we have $\tilde x \in \mathrm{R}'$ for all $\tilde x \leq_{\mathrm{D}} x$. 
Analogously for $(x,y) \in r_2$ we have $(x,y) \in \mathbb{R} \times \mathbb{R} \seq \mathrm D \times \mathbb{R}$ and also 
$(\tilde x, y) \in R_{\mathrm D}$ for all $\tilde x \in \mathrm{D}$ such that $\tilde x \leq_{\mathrm{D}} x$. 
Note that this especially implies $(\tilde x, y) \in R_D$ for all $\tilde x \in \mathrm{R}'$.

Thus every path stopped at a new running maximum creates a line in the left half of $\mathrm D \times \mathbb{R}$ 
while every path stopped at new running minimum creates a line in the right half of $\mathrm D \times \mathbb{R}$ 
which continues on into the left half creating a structure of Rost barrier type.
Figure \ref{fig:Perkins-Loynes} gives an illustration of how the vh-barrier picture in Figure \ref{fig:Perkins} (B) translates into a barrier in 
$\mathrm D \times \mathbb{R}$ which will be of inverse barrier structure. 
\begin{figure}
\centering
\includegraphics[width = 1\linewidth]{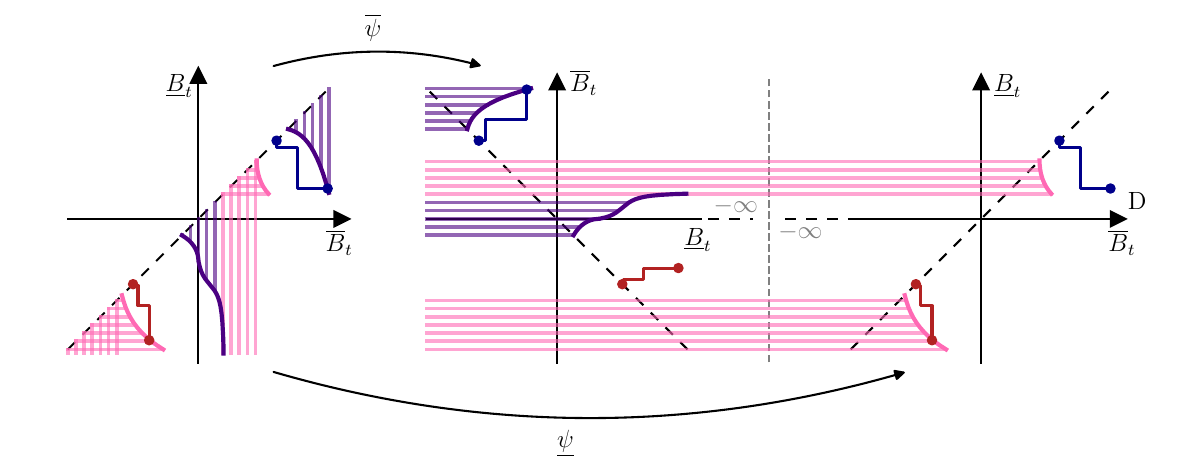}
\caption{Translating the Perkins barrier picture in Figure \ref{fig:Perkins} (B) into $\mathrm D \times \mathbb{R}$.}
\label{fig:Perkins-Loynes}
\end{figure}

We want to consider the path $(\overline B, \underline B)$ in $\mathrm D \times \mathbb{R}$, 
more precisely we want to consider the path on the lhs $\mathrm{R}' \times \mathbb{R}$ as well as on the rhs $\mathbb{R} \times \mathbb{R}$
of $\mathrm D \times \mathbb{R}$ separately but simultaneously.
So we define the following two maps embedding $\mathbb{R}^2$ into the respective space.
\begin{align*}
    &\overline \psi: \mathbb{R}^2 \rightarrow \mathrm{R}' \times \mathbb{R},\quad \overline \psi(x,y) = (\iota(y),x)
\\  &\underline \psi: \mathbb{R}^2 \rightarrow \mathbb{R} \times \mathbb{R},\quad \underline \psi(x,y) = (x,y)
\end{align*}
While $\underline \psi$ will plainly embed the path $(\overline B, \underline B)$ into the right half $\mathbb{R} \times \mathbb{R}$, 
the map $\overline \psi$ will reflect the path $(\overline B, \underline B)$ along the main diagonal before embedding it into the left half $\mathrm{R}' \times \mathbb{R}$. 
As we have `flipped' the order on $\mathrm{R}'$ we will perceive the action of $\overline \psi$ as a counter clockwise rotation by $90^{\circ}$. 
Due to the order defined on $D$ these two paths will now both travel from left to right. 
In the left half the path $\overline\psi\left(\left(\overline B_{t}, \underline B_{t}\right)\right)$ will move vertically upwards whenever a new running maximum is reached and will move only horizontally when a new running minimum is reached. 
Furthermore this implies that the barrier can only be hit when a new running maximum is reached and no stopping will happen in the left hand side due to a new running minimum. 
In the right half the path $\underline\psi\left(\left(\overline B_{t}, \underline B_{t}\right)\right)$ will move vertically downwards whenever a new running minimum is reached but will travel perfectly horizontal at the level of the current running minimum when a new running maximum is reached.
In other words, all stopping due to new running maxima will happen on the left hand side $\mathrm{R}' \times \mathbb{R}$ while all stopping due to new running minima will happen on the right hand side $\mathbb{R} \times \mathbb{R}$.
We consider the respective stopping times of the embedded processes
\begin{align*}
    \overline \tau_R  &:= \inf \left\{t \geq 0 : \overline\psi\left(\left(\overline B_{t}, \underline B_{t}\right)\right) \in R_{\mathrm D} \right\},
\\  \underline \tau_R &:= \inf \left\{t \geq 0 : \underline\psi\left(\left(\overline B_{t}, \underline B_{t}\right)\right) \in R_{\mathrm D} \right\}.
\end{align*}
Now if $R \subseteq \mathbb{R}^2$ and $R_{\mathrm D} \subseteq \mathrm{D} \times \mathbb{R}$ are both induced by the same sets $r_1$ and $r_2$ found in Theorem \ref{thm:Perkins}, then 
\[
    \tau_R = \inf \left\{t \geq 0 : \left(\overline B_{t}, \underline B_{t}\right) \in R \right\} 
           = \overline \tau_R \wedge \underline \tau_R,
\]
thus both representations of the hitting time of the barrier will be equivalent. 
We will assume $R_{\mathrm D}$ to be finely closed with respect to $\overline\psi\left(\left(\overline B_{t}, \underline B_{t}\right)\right)$ as well as $\underline\psi\left(\left(\overline B_{t}, \underline B_{t}\right)\right)$ and give a Loynes type uniqueness argument for the stopping time $\tau_R$.
\begin{thm} \label{thm:advancedLoynes}
Let $R, S \subseteq \mathbb{R}^2$ be two vh-barriers and let $\tau_R$ resp. $\tau_S$ denote their hitting times by the process $\left(\overline B_{t}, \underline B_{t}\right)$. 
If $\tau_R$ and $\tau_S$ both embed the same law $\mu$, then $\tau_R = \tau_S$ a.s. 
\end{thm}
\begin{proof}
Note that by the above discussion we must have inverse barriers $R_{\mathrm D}, S_{\mathrm D} \subseteq \mathrm{D} \times \mathbb{R}$ corresponding to the vh-barriers $R$ resp.\ $S$ 
such that $\tau_R = \overline \tau_R \wedge \underline \tau_R$ and $\tau_S = \overline \tau_S \wedge \underline \tau_S$. 
Analogously to Proposition \ref{prp} and Corollary \ref{cor} the almost sure equality $\tau_R = \tau_S$ will follow by defining a set 
$K \subseteq \mathbb{R}$ of those levels where barrier $R$ is `longer' than barrier $S$ and showing that 
\begin{equation} \label{eq:P-K}
    P\left[ B_{\overline \tau_S \wedge \underline \tau_S} \in K, B_{\overline \tau_R \wedge \underline \tau_R} \in K^{c}\right] = 0.
\end{equation}
Considering the order $\leq_{\mathrm{D}}$ on $\mathrm{D}$ 
as well as the fact that $R_{\mathrm{D}}$ respectively $S_{\mathrm{D}}$ are both inverse barriers we define
\[
    K:= \left\{z \in \mathbb{R} : \sup\left\{ x \in D: (x,z) \in R_{\mathrm{D}} \right\} >_{\mathrm{D}} \inf\left\{x \in D : (x,z) \in S_{\mathrm{D}} \right\}\right\}.
\]
Thus now we can consider the subset
\begin{align*}
    R^K := \left\{(x,y) \in R_{\mathrm{D}}: y \in K\right\} \subseteq R_{\mathrm{D}}
\end{align*}
and define $R^{K^c}, S^K$ and $S^{K^c}$ analogously.
Then we have $S^K \subseteq R^K$ and $R^{K^c} \subseteq S^{K^c}$.

Consider $B_{\overline \tau_S \wedge \underline \tau_S} \in K$, 
then it is crucial to observe that our stopping times will only stop  Brownian motion at a \emph{new} running minimum or running maximum, 
hence by definition of the respective stopping times we will either have 
$B_{\overline \tau_S \wedge \underline \tau_S} = \overline B_{\overline \tau_S}$ or 
$B_{\overline \tau_S \wedge \underline \tau_S} = \underline B_{\underline \tau_S}$, 
thus either 
\begin{align*}
    \overline\psi\left(\left(\overline B_{t}, \underline B_{t}\right)\right)  \in S^K \subseteq R^K 
    \quad  \text{ or } \quad
    \underline\psi\left(\left(\overline B_{t}, \underline B_{t}\right)\right) \in S^K \subseteq R^K.
\end{align*}
Assume $B_{\overline \tau_S \wedge \underline \tau_S} = \overline B_{\overline \tau_S}$. 
Then $\overline \tau_S \leq \underline \tau_S$ and $\overline \tau_R \leq \overline \tau_S$ by the definition of $K$, 
hence $\overline \tau_R \wedge \underline \tau_R \leq \overline \tau_S \wedge \underline \tau_S$ and it is impossible to have
$B_{\overline \tau_R \wedge \underline \tau_R} \in K^{c}$ as otherwise $B_{\overline \tau_S \wedge \underline \tau_S}$ would have stopped in $K^{c}$ as well. 
We have the analogous result when $B_{\overline \tau_S \wedge \underline \tau_S} = \underline B_{\underline \tau_S}$, 
hence \eqref{eq:P-K} is clear.
We can conclude the proof analogously to the proof of Proposition \ref{prp} and Corollary \ref{cor}.
\end{proof}

We can now complete the proof of Theorem \ref{thm:maintheorem}.
\begin{proof}
For an arbitrary choice of $\varphi$ consider $\tau_P:= \tilde \tau$ as found in Theorem \ref{thm:Perkins}. 
Let $\overline \tau$ be another stopping time constructed as in Theorem \ref{thm:Perkins} but for a different choice of $\varphi$. 

First note that the behaviour in $0$ is independent of $\varphi$, 
hence for $A \in \mathcal B (\mathbb{R})$ we have
\[
    \P \left[ \tau_{P} = 0, B_0 \in A \right] = (\lambda \wedge \mu) (A) = \P\left[ \overline \tau = 0, B_0 \in A \right].
\]
Now note that on $\{\tau_{P} > 0 \} = \{\overline \tau > 0 \}$ both stopping times are given as hitting times of vh-barriers and embed the same law $\mu$, 
hence by Theorem \ref{thm:advancedLoynes} we have $\tau_{P} = \overline \tau$ a.s. 
Since the choices for $\varphi$ were arbitrary we must have optimality for any such $\varphi$, thus we can conclude that $\tau_R$ minimizes resp.\  maximizes $\overline B_{\tilde\tau}$ resp.\  $\underline B_{\tilde\tau}$ in law.
\end{proof}
%
%
%
%
%
%
\subsection{An Example}
We will illustrate the occurrences of the different type of lines of the Perkins barrier solution in the following simple atomic example. 
Consider the initial measure
\[
\lambda = \frac{1}{4} \delta_{-1} + \frac{1}{2} \delta_{0} + \frac{1}{4} \delta_{1},
\]
and for $\alpha \in [0,1]$ the terminal measure
\[
\mu = \frac{1-\alpha}{2} \delta_{-2} + \alpha \delta_{0} + \frac{1-\alpha}{2}  \delta_{2}.
\]
Depending on the size $\alpha$ of the atom in $\{0\}$ we have the following different cases.
\begin{itemize}
\item[{$\alpha \in \left[0,\frac{1}{2}\right]$:}] 
The measure $\lambda$ and $\mu$ share mass $\mu(\{0\})=\alpha \leq \frac{1}{2} = \lambda(\{0\})$ in $\{0\}$, 
thus $(\lambda \wedge \mu)(\{0\}) = \alpha$ and all mass needed in $\{0\}$ will get stopped immediately while all other mass will be left to run to $\pm2$ as illustrated in Figure \ref{fig:example-1}.
\item[{$\alpha \in \left(\frac{1}{2},\frac{5}{8}\right]$:}] 
Since $\mu(\{0\}) = \alpha > \frac{1}{2} = \lambda(\{0\})$ we have $(\lambda \wedge \mu)(\{0\}) = \frac{1}{2}$. 
Hence all mass available in $\{0\}$ will be stopped immediately but then still a little more mass is needed. 
An additional v-line in $0$ where some of the paths starting in $B_0 = -1$ will stop upon reaching a new running maximum in $0$ will provide this missing mass as illustrated in Figure \ref{fig:example-2}.
Note that by adding this additional v-line we can at maximum acquire an additional $\frac{1}{8}$ of mass in $\{0\}$, 
hence resulting in a maximal atom size of $\frac{5}{8}$ to be embedded this way.
\item[{$\alpha \in \left(\frac{5}{8},\frac{3}{4}\right]$:}] 
As seen in the previous case we have $(\lambda \wedge \mu)(\{0\}) = \frac{1}{2}$ and all mass available in $\{0\}$ will be stopped immediately. 
However as discussed above merely adding a v-line will no longer suffice to acquire the $\alpha > \frac{5}{8}$ atom in $\{0\}$. 
Hence we also need to stop (some) paths starting in $B_0 = 1$ upon reaching a new running minimum in $\{0\}$ via an h-line in $\{0\}$ as illustrated in Figure \ref{fig:example-3}.
Note that by adding this horizontal line we can at most acquire an additional $\frac{1}{8}$ of mass in $\{0\}$ from the paths started in $B_0 = 1$, 
adding up to a total atom size of $\frac{3}{4}$ that could be embedded this way.
\item[{$\alpha \in \left(\frac{3}{4}, 1\right]$:}] The measures $\lambda$ fails to be prior to $\mu$ in the convex order when $\alpha > \frac{3}{4}$. 
While technically we could provide vh-barriers such that the corresponding hitting time would embed the measure $\mu$, 
the resulting stopping time would no longer be integrable hence not be a solution to (SEP$_{\lambda, \mu}$).
\end{itemize}
\begin{figure}
\centering
\includegraphics[width = 1\linewidth]{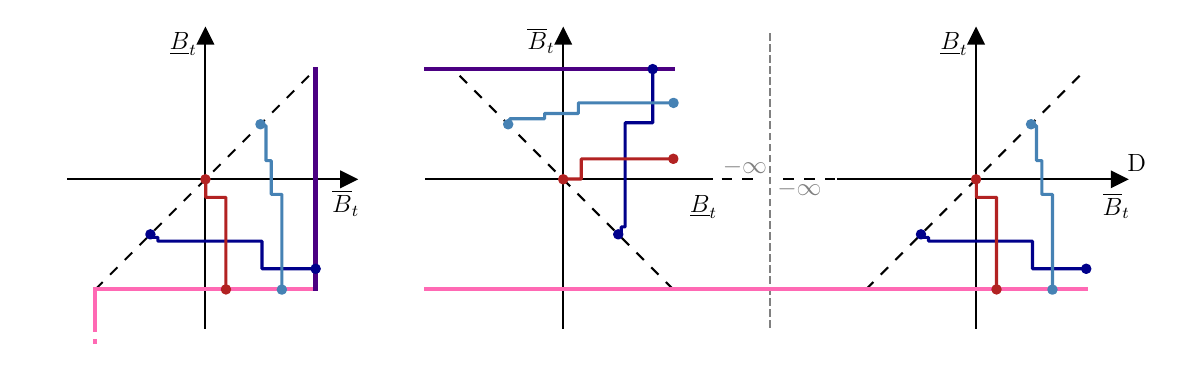}
\caption{The Perkins embedding of $\mu$ for $\alpha \in \left[0,\frac{1}{2}\right]$.}
\label{fig:example-1}
\end{figure}
\begin{figure}
\centering
\includegraphics[width = 1\linewidth]{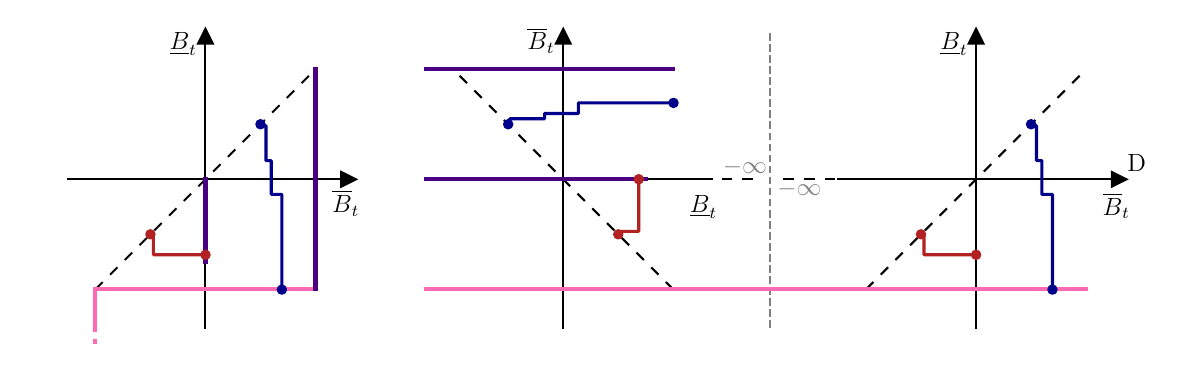}
\caption{The Perkins embedding of $\mu$ for $\alpha \in \left(\frac{1}{2},\frac{5}{8}\right]$.}
\label{fig:example-2}
\end{figure}
\begin{figure}
\centering
\includegraphics[width = 1\linewidth]{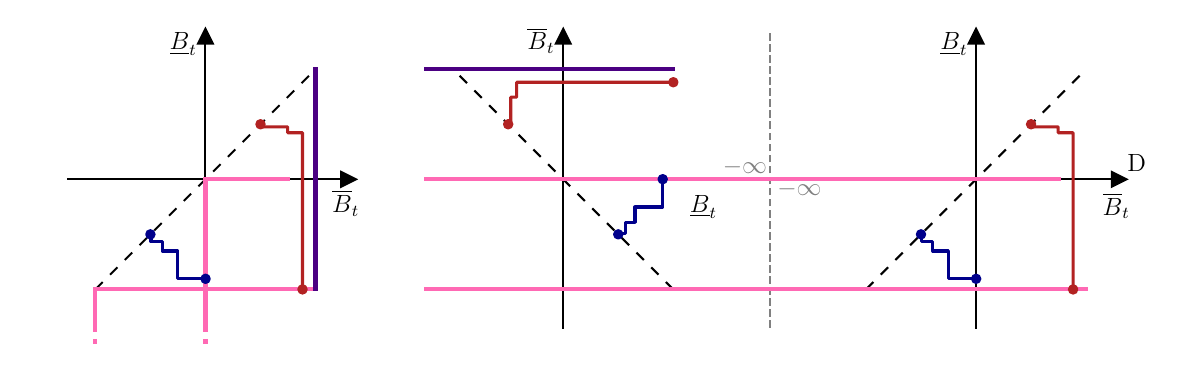}
\caption{The Perkins embedding of $\mu$ for $\alpha \in \left(\frac{5}{8},\frac{3}{4}\right]$.}
\label{fig:example-3}
\end{figure}
%
%
%
%
%
%
%
%
%
%
\bibliographystyle{plain}
\bibliography{joint_biblio}
\end{document}